\documentclass[11pt,reqno]{amsart}
\usepackage{amsmath, amssymb, amsthm}

\usepackage{epsfig}
\usepackage{graphicx}
\usepackage{color}
\usepackage[vcentermath]{youngtab}
\usepackage{url}

\usepackage{wrapfig}
\usepackage{amssymb}
\usepackage[colorlinks=true,linkcolor=black,citecolor=black,urlcolor=blue]{hyperref}

  \theoremstyle{plain}
  \newtheorem{lem}{Lemma}
  \theoremstyle{plain}
  \newtheorem*{thm*}{Theorem}
  \theoremstyle{plain}
 \newtheorem*{conj*}{Conjecture}
\newtheorem*{ques*}{Question}

\theoremstyle{plain}
\newtheorem{theorem}{Theorem}

\theoremstyle{definition}
\newtheorem{definition}[theorem]{Definition}

\newtheorem{corollary}[theorem]{Corollary}

\DeclareMathOperator{\sign}{sign}

\DeclareMathOperator{\GL}{GL}

\DeclareMathOperator{\wt}{Wt}
\DeclareMathOperator{\id}{id}
\DeclareMathOperator{\tr}{tr}
\DeclareMathOperator{\sgn}{sgn}
\DeclareMathOperator{\str}{Star}
\DeclareMathOperator{\height}{height}

\newcommand{\RR}{{\mathbb R}}

\newcommand{\N}{\mathbb{N}}

\newcommand{\R}{\mathbb{R}}
\newcommand{\C}{\mathbb{C}}
\newcommand{\bS}{\mathbf{S}}
\newcommand{\bA}{\mathbf{A}}
\newcommand{\bB}{\mathbf{B}}
\newcommand{\sgnrep}{\mathbf{sgn}}

\usepackage{graphicx}
\usepackage{color}

\begin{document}

\title{Ordering the representations of $S_n$ using the interchange process}

\author{Gil Alon and Gady Kozma}

\address{Gady Kozma\\
Department of Mathematics,\\
Weizmann Institute of Science,\\
Rehovot 76100, Israel}
\email{gadyk@weizmann.ac.il}

\address{Gil Alon\\
Division of Mathematics,\\
The Open University of Israel,\\
1 University Road, Raanana 43107, Israel}
\email{gilal@openu.ac.il}

\begin{abstract}
Inspired by Aldous' conjecture for
the spectral gap of the interchange process and its recent
resolution by Caputo, Liggett and Richthammer \cite{CLR09}, we define
an associated order $\prec$ on the irreducible representations of $S_n$. Aldous'
conjecture is equivalent to certain representations being comparable
in this order, and hence determining the ``Aldous order'' completely is a
generalized question. We show a few additional entries for this order.
\end{abstract}



\maketitle

\section{Aldous' order}
Let $G$ be a finite graph with vertex set $\{1,\dotsc,n\}$, and equip each edge $\{i,j\}$ with an alarm clock
that rings with exponential rate $a_{i,j}$. Put a marble in every
vertex of $G$, all different, and whenever the clock of $\{i,j\}$ rings,
exchange the two marbles. Each marble therefore does a standard
continuous-time random walk on the graph but the different walks are
dependent. This process is called {\em the interchange process} and is
one of the standard examples of an interacting particle system,
related to exclusion processes (where the marbles have only a few
possible colors) but typically more complicated. Further, when one
considers the evolution of the permutation taking the initial
positions of the marbles to their positions at time $t$, one gets a
continuous-time random walk on a (weighted) \emph{Cayley graph} of the
group of permutations $S_n$.

The first landmark in the understanding of this process was the work
of Diaconis \& Shahshahani \cite{DS81}. For the case of $G$ being the
complete graph they diagonalized the relevant $n!\times n!$ matrix
completely using representation theory and achieved very fine results
on the mixing properties.

If one cannot get the whole spectrum, the second eigenvalue (so-called
the spectral gap) allows to get significant partial information on the
process. In 1992 Aldous made the bold conjecture that the spectral gap
of the interchange process is in fact equal to the spectral gap of the
simple random walk on $G$, for every $G$\label{pg:aldous}. This was
the focus of much 
research \cite{B94,HJ96,M08,SC08,C09a,D09} and was finally resolved by
Caputo, Liggett and Richthammer \cite{CLR09}. Our focus in this paper
is however the spectrum as a whole, and for this we need to discuss
the problem from a representation theoretical point of view. More details
on representation theory will be given below in \S\ref{sec:rep}, for
now we continue assuming the reader has basic familiarity with the
subject.  

Let $n\in\N$ and let
$A=\left\{a_{i,j}\right\}_{1\le i<j\le n}$ with all $a_{i,j}$
non-negative. Examine the following {\em formal sum of permutations with real
  coefficients}\footnote{Or element of the group ring $\RR[S_n]$, if you prefer}
\[
\Delta_A=\sum_{i<j}a_{i,j}(\id-(ij))
\]
where id stands for the identity permutation. Let $\rho$ be
any representation of $S_n$. Then 
\begin{equation}\label{eq:defrhoDelta}
\rho(\Delta_A)=\sum_{i<j}a_{i,j}\big(\rho(\id)-\rho((ij))\big)
\end{equation}
is some $\dim\rho\times\dim\rho$ matrix. It is well known that
$\rho(\Delta_A)$ is a positive-semidefinite matrix, indeed $(ij)$ is an
involution so all the eigenvalues of $\rho((ij))$ are $\pm 1$ and the
eigenvalues of $\rho(\id-(ij))$ are in $\{0,2\}$, so each term in
(\ref{eq:defrhoDelta}) is positive and hence so is their sum. We shall
denote the eigenvalues of $\rho(\Delta_A)$ by
\[
\lambda_1(A;\rho)\le\dotsb\le\lambda_{\dim(\rho)}(A;\rho)
\]
We will occasionally drop the $A$ from the notation. 

Now, the irreducible representations of $S_n$ are indexed
by partitions of $n$. Namely, for each integer sequence $r_1\ge r_2\ge 
\dotsb\ge r_k>0$ with $\sum_{i=1}^k r_i=n$ there exists a unique
irreducible representation, which we shall denote by
$[r_1,\dotsc,r_k]$. Such a partition has a nice graphical
representation given by the associated \emph{Young diagram} obtained
by drawing each $r_i$ as a line of boxes from top to bottom, 
\[
[5,1]={\tiny\yng(5,1)}\qquad [3,2,1]={\tiny\yng(3,2,1)} \qquad [2,1^3]={\tiny\yng(2,1,1,1)}
\]
and we will occasionally use it. We may now state Aldous' conjecture
\begin{thm*}[Caputo, Liggett \& Richthammer \cite{CLR09}] For any $A$
  and any irreducible $\rho$ different from the trivial representation $[n]$,
\begin{equation}\label{eq:Aldous}
\lambda_1(A;[n-1,1])\le\lambda_1(A;\rho).
\end{equation}
\end{thm*}
\noindent The formulation of the result in \cite{CLR09} does not use
representation theory. As discussed on page \pageref{pg:aldous}, they
showed that the interchange process has the same spectral gap as
the simple random walk. See \cite{C09a} for how to get from one
formulation to the other. 

Faced with (\ref{eq:Aldous}) one is
tempted to generalize the question. Define the {\bf Aldous order} on
irreducible representations by
\[
\rho\preceq\sigma \iff \forall A\; \lambda_1(A;\rho)\ge\lambda_1(A;\sigma)
\]
where again by $\forall A$ we mean for all $n\times n$ matrices with
non-negative coefficients. It is rather unfortunate that the
\emph{largest} representation in the $\prec$ order has the
\emph{smallest} $\lambda_1$, but we wish to make $\prec$ consistent
with the \emph{domination order}, for which there is already an established
direction. We say that $\rho\prec\sigma$ if $\rho\preceq\sigma$ and
they are different, and remark that $\rho\preceq\sigma$ and
$\rho\succeq\sigma$ imply that $\rho=\sigma$, see page
\pageref{pg:precsucceq}, remark 1.

As can be seen in figure \ref{fig:ao}, $\prec$ is an interesting object,
and it seems to be correlated with the domination order
$\vartriangleleft$. We say that $\sigma\vartriangleleft\rho$
if $\sigma$ can be obtained from $\rho$ by a sequence of steps, such that
in each step one box is dropped to the row below it, in a way that
leaves a Young diagram. Cesi \cite{C09a} remarked that it would be
nice if Aldous' order were identical to the domination order, but also
noted a counterexample in $n=4$: one has
${\tiny\yng(2,2)}\vartriangleright{\tiny\yng(2,1,1)}$ but
${\tiny\yng(2,2)}\nsucc{\tiny\yng(2,1,1)}$. See
\cite[Counterexample 8.2]{C09a}. Such a counterexample exists for
every $n\ge 4$: we will see below in corollary \ref{cor:asympval},
page \pageref{cor:asympval}, that $[2,2,1^{n-4}]\nsucc
[2,1^{n-2}]$, although clearly $[2,2,1^{n-4}] \vartriangleright [2,1^{n-2}]$. The graph demonstrating this is a star. Shannon Starr
discovered numerically another asymptotic family of counterexamples
for even sizes, $[n+1,n-1]\nsucc[n,n]$ (private communication, checked
numerically up to size 14, $n=7$). In this case the graph
demonstrating it is the cycle. We remark also that the results of
Diaconis \& Shahshahani \cite{DS81} imply that if $\sigma\vartriangleleft\tau$
then $\sigma\nsucc\tau$. We do not know if it is generally true that
if two representations are $\vartriangleleft$-incomparable then they are
also $\prec$-incomparable, so we cannot quite state that $\prec$ is a
sub-order of $\vartriangleleft$, though it is a very natural conjecture.
\begin{figure}
\includegraphics{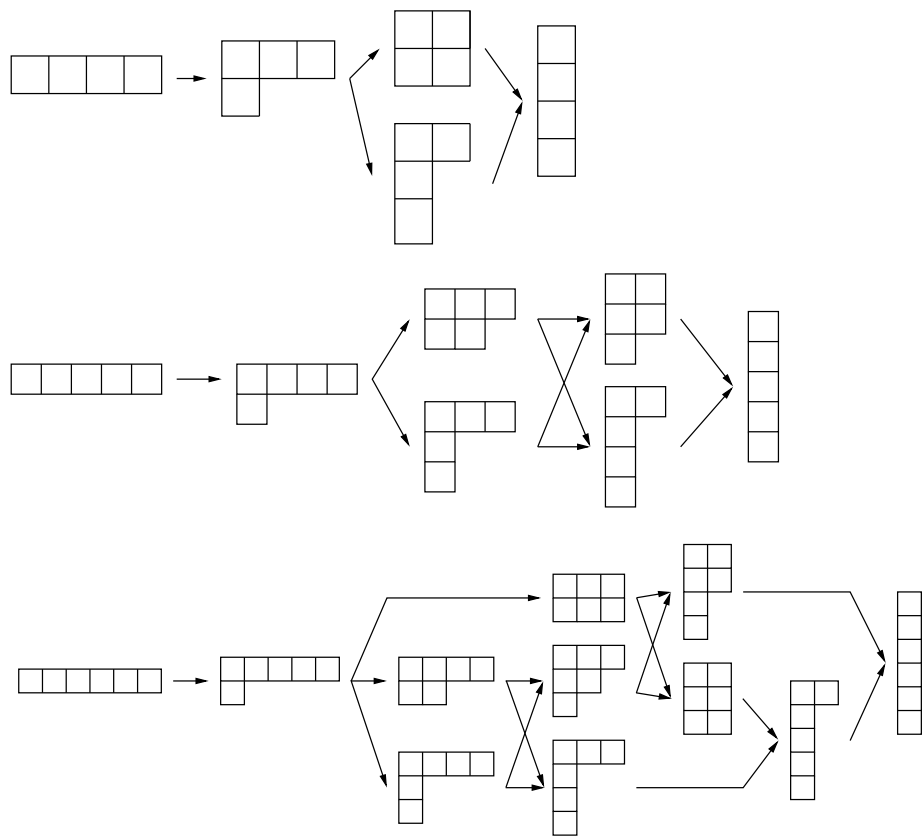}
\caption{\label{fig:ao}The Aldous order for $n=4,5,6$ Arrows are drawn
  from the larger representations to the smaller ones. For $n=4$
  everything is proved, but the others are results of computer
  simulations. What can be trusted in the diagrams are the {\em
    non-arrows} --- if the arrows do not imply a relationship between
  two diagrams that means a computer search found two examples proving
  no relationship may exist.}
\end{figure}

Despite the similarity with the completely explicit
$\vartriangleleft$, it is not easy to prove any entry in the Aldous
order. In fact, the only entries in the Aldous order
known previous to this paper are $[n]\succeq\rho\succeq [1,\dotsc,1]$
for all $\rho$ (this is easy once one identifies these
representations, see corollary \ref{cor:n1n}, page \pageref{cor:n1n}), a result of Bacher \cite{B94} that the
hook-shaped diagrams are ordered among themselves
\[
[n]\succ [n-1,1]\succ [n-2,1^2]\succ \dotsb \succ [2,1^{n-1}]
\succ [1^n],
\]
(we fill some details about this in the appendix, page
\pageref{sec:gamma}) and of course the Caputo et al.\ result,
$[n-1,1]\succ\rho$ for all $\rho\ne [n],[n-1,1]$. 

We may now state our result.
\begin{thm*}\label{thm:ao}
Let $n\ge 4k^2+4k$. Let $\tau$ be an irreducible representation whose Young
diagram has $\ge n-k$ boxes in the first row; and let $\sigma$ be an
irreducible 
representation whose Young diagram has $\ge n-k$ boxes at the 
the leftmost column. Then $\tau\succ\sigma$.
\end{thm*}
\noindent Again we see the relationship with $\vartriangleright$. What we show
is that if ``$\tau \vartriangleright\!\!> \sigma$'' i.e.\ if $\tau$
is much larger in the domination order than $\sigma$, then
$\tau\succ\sigma$.

Let us end this introduction returning to the work of Diaconis \&
Shahshahani and to the mixing time of Markov chains. Better
understanding of Aldous' order will allow to extend their results to
other graphs. Of particular interest are the hook-shaped
representations because for them all eigenvalues are explicitly known
\cite{B94}. Determining which representations are $\prec$ than a given
hook-shaped representation will allow to quickly estimate their
contribution to the mixing of the process. Let us formulate some modest
questions
\begin{ques*}
Describe the representations $\sigma \prec[n-2,1^2]$. The natural
generalization of Aldous conjecture is that $\sigma
\vartriangleleft [n-2,1^2]$ implies $\sigma\prec[n-2,1^2]$. Is this
true? If not, maybe there exists some absolute constant $K$ such that any
$\sigma\vartriangleleft [n-K,K-1,1]$ satisfies $\sigma\prec [n-2,1^2]$?
\end{ques*}
There are of course many other natural questions about this order. How
many entries does it have? What is the longest chain? Is it really a
subset of the domination order? What
is the longest chain in the domination order which is completely
uncomparable in Aldous order? etc. The simulation results seem to
indicate that $\sigma\succ\rho\;\triangleright\;\tau$ implies
$\sigma\succ\tau$. We see no particular reason for this to be true,
but if it is it would be interesting. We chose to highlight the question
above because we believe it has relevance to questions which do not
need representation theory to state, e.g.~mixing time and the
quantum Heisenberg ferromagnet (see \cite{AK10} for the latter).

\section{Some representation theory}\label{sec:rep}
This section will contain only the minimal set of facts needed for the
paper. For a thorough introduction to the topic see one of the books
\cite{FH91,G93,JK81,S00} and the influential paper \cite{OV96}.
A representation of $S_n$ is a group homomorphism $\rho:S_n\to \GL(V)$ where
$V$ is some linear space over $\C$ and $\GL(V)$ is the space of all
linear transformations of $V$. We will assume throughout that $V$ is
finite dimensional. It can be assumed \cite[Theorem 1.5.3]{S00} that
$\rho(g)$ is a unitary matrix, and so we will assume this a-priori for
all our representations. We denote $\dim\rho=\dim V$.

Given two representations $\rho_i:S_n\to\GL(V_i)$, $i=1,2$ one may
construct their direct sum $\rho_1\oplus\rho_2:S_n\to\GL(V_1\oplus
V_2)$ by 
\[
(\rho_1\oplus\rho_2)(g)=\left(\begin{array}{cc}
\rho_1(g) & 0\\
0 & \rho_2(g)
\end{array}\right).
\]
On the other hand, if there is a decomposition $V=V_1\oplus V_2$ such that
for any $g\in S_n$, $\rho(V_i)\subset V_i$ for $i=1,2$ then one may
construct $\rho_i(g)=\rho(g)|_{V_i}$ and get
$\rho\cong\rho_1\oplus\rho_2$. If no such decomposition exists we say
that $\rho$ is irreducible. Every representation can be written as a
direct sum of irreducible representations, and the isomorphism classes of
the factors are unique up to order
\cite[Proposition 1.7.10]{S00}. Recall also Schur's lemma which states
that a linear map from an irreducible $V$ to $V$ which commutes with
the action of every $g\in S_n$ is a constant multiple of the identity
\cite[Corollary 1.6.8]{S00}.

\subsection{Young diagrams}
We will use one specific method that constructs all the irreducible
representations as explicit subspaces of the group ring $R$. The
construction is somewhat abstract, but we will only need a few
properties which will be easy to deduce, and we will do so in lemma
\ref{lem:Q} and (\ref{eq:reverse}) below and forget about the actual
definition of the representations. 

Recall that the group ring $\RR[S_n]$ is simply the collection of all formal
sums $\sum_{g\in S_n} a_g g$ with coefficients $a_g\in\RR$. We will
denote $R=\RR[S_n]$. $S_n$ acts
on $R$ by 
\[
h\left(\sum a_g g\right)=\sum a_g hg
\]
which makes $R$ into a (left) representation known as the 
\emph{regular representation}. It is true generally for any finite group
that any irreducible representation can be embedded into the regular
representation \cite[Proposition 1.10.1]{S00}. For a general finite
group this requires to work over $\C$, but as we will see shortly, the
specific structure of the representation of $S_n$ allows to work over
$\RR$, which is more natural in our setting.

Let now $\tau_1\ge\tau_2\ge\dotsb\ge\tau_m$ with $\sum\tau_i=n$. We define $H$ to be the
group of permutations $h$ that preserve the rows of the diagram
$[\tau_1\dotsc,\tau_m]$ in the sense that 
\begin{equation}\label{eq:DefH}
i\in [1,\tau_1] \Rightarrow h(i)\in[1,\tau_1];\quad
i\in [\tau_1+1,\tau_1+\tau_2] \Rightarrow
h(i)\in[\tau_1+1,\tau_1+\tau_2]
\quad \mbox{etc.}
\end{equation}
Let $V$ be the group of permutations that preserve the columns of the
diagram $[\tau_1,\dotsc,\tau_m]$ e.g.~any $v\in V$ must preserve the
set $\{1,\tau_1+1,\tau_1+\tau_2+1,\dotsc,n-\tau_m+1\}$. We define the following
elements of the group ring $R$,
\[
a_\tau=\sum_{h\in H}h \quad b_\tau=\sum_{v\in V}\sign(v)v\quad
c_\tau=a_\tau b_\tau\,.
\]
Then the representation $[\tau_1,\dotsc,\tau_m]$ is
defined to be $Rc_\tau=\{rc_\tau:r\in R\}$, with the group acting by
multiplication from the left. This is a subspace of $R$
which is easily seen to be closed under the action of $S_n$. By
\cite[Theorem 4.3]{FH91} these representations are
irreducible and exhaust all the irreducible representations of $S_n$.

To shed a little light on the definition let us take two
examples. The first is $[n]$. In this case $H=S_n$ and
$V=\{\id\}$. Hence $c_\tau=\sum_{g\in S_n}g$ so $hc_\tau=c_\tau$ for
any $h\in S_n$. This means that $[n]$ is one-dimensional with a
trivial action of $S_n$. This representation is also known as the
\emph{trivial representation}. A second example is $[1^n]$. In
this case $H=\{\id\}$ and $V=S_n$ so this time $c_\tau=\sum_{g\in
  S_n}\sign(g)g$. We get that $hc_\tau=\sign(h)c_\tau$ so this
representation is also one-dimensional, but this time the action of
$S_n$ is by multiplication with the sign of the permutation, the
so-called \emph{sign representation}, which we will denote by
$\sgnrep$. In general, if $\tau$ is any Young diagram and if $\tau'$ is
the diagram one gets by reflecting $\tau$ along the main diagonal (so
that the lengths of the rows of $\tau$ become the lengths of the columns
of $\tau'$) then
\begin{equation}\label{eq:reverse}
[\tau']=[\tau]\otimes\sgnrep\,.
\end{equation}
See \cite[2.1.8]{JK81}.

Now, $Ra_\tau$ is also a representation. Generally it is reducible,
but it is much more convenient to work with. Indeed, $ga_\tau$ is
simply $\sum_{h\in gH}h$ so $Ra_\tau$ is isomorphic to the natural
action of $S_n$ on the set of cosets $\{gH:g\in S_n/H\}$. Further,
each coset can be thought of as a coloring of $n$ by $m$ colors with
exactly $\tau_1$ numbers colored in the first color, exactly $\tau_2$
numbers colored in the second color, etc. Formally we define
\[
Q=Q(\tau)=\{q:\{1,\dotsc,n\}\to\{1,\dotsc,m\}:\#q^{-1}(i)=\tau_{i}\}\,.
\]
and let $L^2(Q)$ be a representation of $S_n$ with the natural
action $(gq)(i)=q(g^{-1}i)$. We will mainly work with these representations, and we relate
them to the irreducible ones by
\begin{lem}\label{lem:Q}
Let $\sigma_1\ge\dotsc\ge\sigma_m$ with $\sum \sigma_i=n$. Then 
\begin{enumerate}
\item $[\sigma_1,\dotsc,\sigma_m]$ can be embedded in $L^2(Q(\sigma))$.
\item For any $q\in Q$ there is a non-zero element of this embedding which is
invariant under any permutation $\phi$ that preserves the coloring
$q$ i.e.~to any $\phi$ for which $q(\phi(i))=q(i)$ for all $i$.
\end{enumerate}
\end{lem}
For example, for $[n-1,1]$ we have that $m=2$ and an element $q\in Q$
is uniquely 
identified by $q^{-1}(2)$ which is an element of
$\{1,\dotsc,n\}$. Hence $|Q|=n$ and $L^2(Q)$ can be though of as
$\RR^n$ with $S_n$ acting by permutation matrices
(this representation is known as the \emph{standard
  representation} \label{pg:stdrep} of
$S_n$). Clearly the constant vectors form a one-dimensional invariant
subspace, and so is their orthogonal complement, the vectors whose
entries sum to 0. It is not difficult to see (directly from the
definition) that both are irreducible representations. The first is
the trivial, hence the second is $[n-1,1]$. 

Now, the second clause of the lemma in this example is as follows. Take some $q\in Q$
i.e.~$q(i)=1$ for all $i\in\{1,\dotsc,n\}$ except one $k$ for which
$q(k)=2$. A permutation $\phi$ preserves $q$ if and only if
$\phi(k)=k$. An element of $L^2(Q)$ invariant to any such $\phi$ must
be constant on $\{1,\dotsc,n\}\setminus\{k\}$, and the only such
element (up to multiplication by constants) in the subspace isomorphic to $[n-1,1] $ is 
$(1,\dotsc,1,1-n,1\dots,1)$ where the position of the negative entry
is $k$ (we are not interested in the uniqueness, only in the existence).

We will prove lemma \ref{lem:Q}
immediately after this simple claim
\begin{lem} \label{lem:sum} Let $\rho:S_n\to\GL(V)$ be a representation and let 
$V_1,\dotsc,V_m$ be subspaces of $V$ invariant under the action of
$S_n$. Let $W$ be an irreducible component of $\sum V_i$. Then $W$
is isomorphic to a component of one of the $V_i$.
\end{lem}
\begin{proof} Clearly it is enough to prove this for just two
  subspaces $V_1$ and $V_2$. Denote $U=V_1\cap V_2$. Then $U$ is
  invariant under the action of $S_n$. Since every invariant subspace is
  complemented \cite[Proposition 1.5.2]{S00} we can write $V_i=U\oplus
  U_i$ and $V_1+V_2 = U\oplus U_1\oplus U_2$. The lemma now follows by
  the uniqueness of decomposition into irreducible representations.
\end{proof}

\begin{proof}[Proof of lemma \ref{lem:Q}] Examine 
\[
Ra_\sigma R=\sum_{g\in S_n} Ra_\sigma g\,.
\]
It is easy to check that each $Ra_\sigma g$ is a representation which is
isomorphic to $Ra_\sigma$ ($g$ is an invertible element of the group
ring $R$). Since $[\sigma]\cong Rc_\sigma\subset Ra_\sigma R$
we get by lemma \ref{lem:sum} that $[\sigma]$ can be embedded into
$Ra_\sigma$. By the discussion before the statement of the lemma,
$Ra_\sigma\cong L^2(Q)$ so the first claim of lemma \ref{lem:Q} is
proved.

For the second claim, note that the permutations $\phi$ as above form
a group, which we will denote by $H_q$. Now, it is not important which $q$
one takes, since if $v$ is invariant to the action of $H_q$ then 
$gv$ is invariant to the action of $gH_qg^{-1}$; and
$gH_qg^{-1}=H_{gq}$ which can give any $q'\in Q$. So we will verify
the claim for the $H$ defined in (\ref{eq:DefH}). But in this case it
is clear that $c_\sigma$ itself is invariant to the action of $H$. Since
the property of existence of a vector invariant to $H$ is an abstract
property of a representation, then it is not important that the
vector $c_\sigma$ is not necessarily in $Ra_\sigma$ but just in some
isomorphic copy. The lemma is thus proved.
\end{proof}

It is interesting to note that the irreducible components of $L^2(Q)$
are known and are all $[\tau]$ with $\tau\trianglerighteq\sigma$
\cite[Lemma 2.1.10]{JK81}. But we will not use it.

\section{Star graphs and the Gelfand-Tsetlin basis}

For $n\geq k \geq 1$, let $K_{n,k}$ be the graph with vertices
$\{1,\dotsc,n\}$ where all the vertices $\{1,\dotsc,k\}$ are connected
with each other, and the remaining $n-k$ vertices are isolated. By
abuse of notation, we will also denote by $K_{n,k}$ the adjacency
matrix of this graph. For any matrix $A$ we denote by $\wt(A)=\sum_{i<j}a_{i,j}$. 
We put $\str_{n,k}=K_{n,k}-K_{n,k-1}$, a star graph having the vertex
$k$ connected with each of $1,\dotsc,k-1$. We remark that as elements
of the group ring, $\str_{n,k}$ are known as the Jucys-Murphy elements. 

In this section, we shall find all the eigenvalues of
$\sigma(\Delta_{\str_{n,k}})=\sum \sigma(\id)-\sigma((ij))$ for any
irreducible representation $\sigma$ of $S_n$. These will serve as
useful examples, but more
importantly will be used in the proof of the main theorem in section
\ref{sec:main}. We will use the Gelfand-Tsetlin basis, an idea also
used in \cite{D09}.


Now, the theory of Gelfand-Tsetlin pairs and bases is very deep with
many analogs in different categories (see e.g.~\cite{M06} for a
survey) but we will not need any of it here. For our purposes it is
enough to define the basis inductively as follows: If $n=1$ then the
representation must be one dimensional and we take a nonzero vector as
the basis vector. For $n>1$, we consider the natural embedding
$S_{n-1}\hookrightarrow S_n$ whose image is $\{\pi \in S_n:
\pi(n)=n\}$. We decompose the restriction $\sigma|_{S_{n-1}}$ into
irreducible representations of $S_{n-1}$ as $\bigoplus_i V_i$, take the Gelfand-Tsetlin basis of each $V_i$ and define our basis as the union of the resulting bases.

If $\sigma=[\alpha]$, where $\alpha$ is a Young diagram of size $n$,
then $\sigma|_{S_{n-1}}= \bigoplus_{\beta}[\beta]$, where
$\beta$ goes over all the Young diagrams of size $n-1$, obtained by
removing, in all possible ways, one box from $\alpha$. See \cite[\S 2.8]{S00}.
Hence, the elements of the Gelfand-Tsetlin basis of $[\alpha]$ are in one to one correspondence with the sequences $\alpha=\alpha_n, \alpha_{n-1},\dotsc, \alpha_1 = \tiny\yng(1)$ of young diagrams, in which each $\alpha_i$ for $i<n$ is obtained from $\alpha_{i+1}$ by removing one box. 

\begin{lem} \label{lem9} Let $G=\str_{n,k}$ for some $2\leq k \leq n$, and let $\sigma$ be an irreducible representation of $S_n$. \begin{enumerate} \item The Gelfand-Tsetlin basis of $\sigma$ is a basis of eigenvectors for $e_G$ acting on $\sigma$. \item Let $\alpha=\alpha_n, \alpha_{n-1},\dotsc, \alpha_1 = \tiny\yng(1)$ be a sequence of Young diagrams, where $\alpha_{i-1}$ is obtained from $\alpha_i$ by removing the $x_i$th box at the $y_i$th row, and $\sigma=[\alpha_n]$. Let $v$ be the Gelfand-Tsetlin basis element corresponding to the above sequence. Then, the eigenvalue of $\Delta_G$ with respect to $v$ is $(k-1)+y_k-x_k$. \end{enumerate} \end{lem}
 
\begin{proof}
(1) Let us first prove that each vector in the Gelfand-Tsetlin basis of
$\sigma$ is an eigenvector for each of
$K_{n,1},K_{n,2},\dotsc,K_{n,n}$. We will do it by induction on
$n$. For $n=1$ the claim is vacuous. For $n>1$, we decompose
$\sigma|_{S_{n-1}}$ into irreducible representations and using the
induction hypothesis, conclude that the Gelfand-Tsetlin basis vectors
are eigenvectors of $K_{n,1},\dotsc,K_{n,n-1}$. As for $K_{n,n}$, the
element $\Delta_{K_{n,n}}$ lies in the center of $\R[S_n]$ (being a linear
combination of the sum of all transpositions and the identity). Hence,
by Schur's lemma it acts as a scalar on each irreducible representation of $S_n$. This finishes the inductive step. Hence, each vector of the Gelfand-Tsetlin basis is an eigenvector of $G=K_{n,k}-K_{n,k-1}$.

(2) Let us now find the scalar by which $K_{n,i}$ acts on the
irreducible $S_i$-representation $[\alpha_i]$, assuming that
$\alpha_i$ has row lengths $l_1,\dotsc,l_m$. For that matter, we
invoke the trace formula from \cite[lemma 7]{DS81}, by which, the trace of a transposition acting on $[\alpha_i]$ is 
$$ \frac{\dim [\alpha_i]}{\binom{i}{2}}\sum_{j=1}^m\left (\binom{l_j-j+1}{2}-\binom{j}{2}\right)$$
where we define $\binom{x}{2}=\frac{1}{2}x(x-1)$ also for $x<2$. Hence, $\Delta_{K_{n,i}}=\sum_{1\leq j<k\leq i}(\id-(jk))$ acts on $[\alpha_i]$ via the scalar 
$$ c_i := \wt(K_{n,i})-\sum_{j=1}^m\left (\binom{l_j-j+1}{2}-\binom{j}{2}\right).$$
Let us find the eigenvalue of $\Delta_G=\Delta_{K_{n,k}}-\Delta_{K_{n,k-1}}$ with respect to our basis vector $v$. This eigenvalue is equal to $c_k-c_{k-1}$, and we will now simplify it.
Recall that passing from $\alpha_k$ to $\alpha_{k-1}$ is done by
removing the rightmost box from the $y_k^\textrm{th}$ row. We
distinguish two cases

\emph{Case 1:} $\alpha_k$ and
  $\alpha_{k-1}$ have the same number of rows. In that case, all the
  summands except the $y_k^\textrm{th}$ one cancel out, and we are
  left with 
\begin{align*}
 \wt(\str_{n,k})-\left(\binom{l_{y_k}-y_k+1}{2} - \binom{l_{y_k}-y_k}{2}\right) 
&= \wt(\str_{n,k}) - (l_{y_k}-y_k) \\ &=  k-1-(x_k-y_k).
\end{align*}

\emph{Case 2:} $\alpha_k$ is obtained from $\alpha_{k-1}$ by removing the unique box in the $y_k^\textrm{th}$ row. In that case, we have $l_{y_k}=1$, and the scalar is
\begin{align*} 
\wt(\str_{n,k})&-\left(\binom{l_{y_k}-y_k+1}{2}-\binom{y_k}{2}\right)
=
\wt(\str_{n,k})-\left(\binom{2-y_k}{2}-\binom{y_k}{2}\right)= \\ 
& = \wt(\str_{n,k})-(1-y_k) = k-1-(x_k-y_k).
\end{align*}
Hence, $\Delta_{\str_{n,k}}v= (k-1-(x_k-y_k))v$, and the result follows.
\end{proof} 

\begin{corollary}\label{cor:asympval}For any $n$, $[2,2,1^{n-4}]\nsucc
  [2,1^{n-2}]$.
\end{corollary}
\begin{proof} Examine the star graph $G=\str_{n,n}$. We get that
  $\lambda_1(G;\rho)=\min\{ (n-1)+y-x\}$ where the minimum is taken over
  all boxes $(x,y)$ which can be removed to get a legal Young
  diagram. For $[2,2,1^{n-4}]$ these boxes are $(2,2)$ and (if $n>4$)
  $(1,n-2)$. The minimum is achieved at $(2,2)$ so
  $\lambda_1(G;[2,2,1^{n-4}])=n-1$. For $[2,1^{n-2}]$ the minimum is
  achieved at $(2,1)$ giving that $\lambda_1(G;[2,1^{n-2}])=n-2$. 
\end{proof}
Since $[2,2,1^{n-4}] \vartriangleright [2,1^{n-2}]$ this shows that
$\succ$ and $\vartriangleright$ differ for every $n$. Similarly one
can show that $[2^i,1^{n-2i}]\nsucc[2^j,1^{n-2j}]$ for any $j< i \le n/2$
giving a chain of length $\lfloor n/2\rfloor$ in the domination
order, no two elements of which are comparable in the Aldous order.

\subsection{Quasi-complete graphs} Let $a_2,a_3,\dotsc,a_n$ be non-negative numbers and examine
\[
\sum_{k=2}^n a_k \str_{n,k}
\]
This is not just any combination of stars, but stars formed in a special order, each added vertex connected to all existing ones. We call such graphs quasi-complete graphs. See figure \ref{fig:qc}. 
\begin{figure}
\input{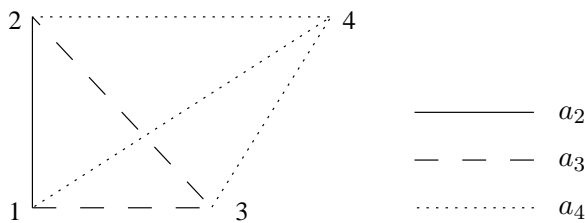}
\caption{\label{fig:qc}A quasi-complete graph.}
\end{figure}
Since the Gelfand-Tsetlin basis of $[\alpha]$ is a basis for of eigenvalues for
$\str_{n,k}$ for all $k$ it is also a basis of eigenvalues for their
linear combination. Further, the basis element corresponding to a sequence $\alpha_n,\dotsc,\alpha_1$ as in lemma \ref{lem9} gives the eigenvalue
\begin{equation}\label{eq:qc}
\wt(G)-\sum_{k=2}^n a_k (x_k-y_k).
\end{equation}
Let us make three remarks about quasi-complete graphs:

1. Taking $a_k$ to be very fast decreasing (taking $a_k=n^{-2k}$ is good enough) it is easy to see that the minimal eigenvalue is achieved when the boxes are removed as follows: first remove the lowest row completely, then the second lowest row completely, etc. This shows that if $\alpha < \beta$ in the lexicographical order then $\alpha \nsucc \beta$. As a corollary we get that $\alpha \preceq \beta$ and $\beta\preceq\alpha$ imply $\alpha=\beta$.\label{pg:precsucceq}

2. This family of graphs is not rich enough to determine the Aldous order.
  For example, take the fact that $[n+1,n-1]\nsucc [n,n]$ which can be verified for small $n\ge 3$ by direct calculation for the circle graph (we have no proof that it holds for all $n$ but this is not relevant at this point). Quasi-complete graphs cannot demonstrate this fact, because any sequence of Young diagrams as in lemma \ref{lem9} for $[n,n]$ must start by removing the box at $(n,2)$, while a sequence for $[n+1,n-1]$ \emph{may} start by removing the box at $(n+1,1)$, and continue by mimicking the first sequence, since in both cases $\alpha_{2n-1}=[n,n-1]$. So we see that
\[
\lambda_1(G;[n+1,n-1])\le \lambda_1(G;[n,n])
\]
for any quasi-complete graph $G$.

3. As an approximation to the Aldous order, one may ask whether two young diagrams $\sigma, \tau$ of size $n$ satisfy $\lambda_1(G; \sigma) \leq \lambda_1(G; \tau)$ for all quasi-complete graphs $G$. Using formula (\ref{eq:qc}),  the answer can be put in terms of the following combinatorial game: Let players $A$ and $B$ get the young diagrams $\sigma$ and $\tau$, respectively. Both players fill out their diagrams as follows: On each square at position $(i,j)$ whey write the number $j-i$.
The game has $n$ steps. At each step of the game, player $B$ breaks off a square from his diagram, in a way that leaves a legal young diagram, and announces the number on that square. Then, player $A$ does the same. We say that player $A$ wins the game if at each step, her number is no less than player $B$'s number.
It is not hard to see that player $A$ has a winning strategy if, and
only if, $\lambda_1(G; \sigma) \leq \lambda_1(G; \tau)$ for all
quasi-complete graphs $G$.


\section{Proof of the main theorem}\label{sec:main}
The proof requires that we examine closely the \emph{maximal}
eigenvalue of $\sigma(\Delta_A)$ (which is of course also its norm as an
$l^2$ operator, as this is a positive matrix). We denote it by
\[
\lambda_{\max}(A;\sigma).
\]
We keep from the previous section the notation
$\wt(A)=\sum_{i<j}a_{i,j}$. For $k<n$,
we denote by $\mathcal{S}_{k}$ the set of representations
corresponding to Young diagrams of size $n$ such
that the first row has $\geq n-k$ boxes (so all the other rows
have, combined, $\leq k$ boxes). Denote 
\[
\mathcal{S}_k\otimes\sgnrep = \{\sigma\otimes\sgnrep :
  \sigma\in\mathcal{S}_k\}
\]
which by (\ref{eq:reverse}) is also the set of representations
corresponding to Young diagrams of size $n$ such that the leftmost
column has $\geq n-k$ boxes.

As in the previous section, we will use ``matrix'' and ``weighted graph''
interchangeably, understanding that all our matrices have non-negative
entries, and that the corresponding graph has an edge at any $(i,j)$
for which $a_{i,j}\ne 0$. In particular when we subtract graphs and
weighted graphs, we are in fact subtracting the corresponding
matrices. Let us start with two standard facts which we prove for the
convenience of the reader.


\begin{lem}
For any $A$ with entries $a_{ij}$ and any representation $\sigma$
\begin{equation}
\lambda_{\max}(A;\sigma)=2\wt(A)-\lambda_1(A;\sigma \otimes \sgnrep). \label{eq:dual}
\end{equation}
and
\begin{equation}
\lambda_{\max}(A;\sigma)\le2\wt(A)\label{eq:triv}
\end{equation}
\end{lem}
\begin{proof}
The action of $\Delta_A=\sum_{i<j}a_{ij}(\id-(ij))$ on $\sigma$ is linearly isomorphic to the action of $$\sum_{i<j}a_{ij}(\id+(ij))=2\wt(A)-\Delta_A$$ on $\sigma \otimes \sgnrep$. This gives (\ref{eq:dual}). The second part follows immediately since $\lambda_1(A;\sigma \otimes \sgnrep) \geq 0$.
\end{proof}

\begin{corollary}\label{cor:n1n} For any $\sigma$,
  $[n]\succeq\sigma\succeq[1^n]$.
\end{corollary}
\begin{proof}Recall that $[n]$ is the trivial representation, so
  $\lambda_1(A;[n])=0$ which shows $[n]\succeq\sigma$. In the other
  direction, $[1^n]=\sgnrep$ so $\lambda_1(A;[1^n])=\lambda_{\max}(A;[1^n])=2\wt(A)\ge\lambda_{\max}(A;\sigma)\ge\lambda_1(A;\sigma)$.
\end{proof}
\begin{lem}
\label{lem:matching}Assume $n\geq4k$ and let $\sigma\in\mathcal{S}_{k}$.
Let $G$ be a graph with $2k$ disjoint edges (i.e.\ $4k$ vertices
have degree $1$, and the remaining $n-4k$ vertices are isolated).
Then\[
\lambda_{\max}(G;\sigma)\leq2k.\]

\end{lem}
\begin{proof} Recall the representation $L^2(Q)$ of lemma \ref{lem:Q},
namely if $\sigma_{1}\geq\dotsc\geq\sigma_{m}>0$ are the lengths of the
rows of (the Young diagram corresponding to) $\sigma$ then
\[
Q=Q(\sigma)=\{q:\{1,\dotsc,n\}\to\{1,\dotsc,m\}:\#q^{-1}(i)=\sigma_{i}\}.
\]
By lemma \ref{lem:Q} we know that the representation $\sigma$ can be
embedded in $L^{2}(Q)$. Hence it is enough to show that
\[
\lambda_{\max}(G;L^{2}(Q))\leq 2k.
\]
Let $f\in L^{2}(Q)$ be an eigenvector for $\lambda_{\max}$, and
let $q\in Q$ be the point where the maximum of $|f|$ is attained.
Fix one edge $(i,j)\in G$ and examine 
\begin{equation}
((\id-(ij))f)(q)=f(q)-f((ij)q).\label{eq:oneij}
\end{equation}
If $q(i)=q(j)=1$ (recall that the elements $q$ of $Q$ are themselves
functions) then $(ij)q=q$ and (\ref{eq:oneij}) is zero. However,
by definition of $Q$ the number of $i$ such that $q(i)\neq1$ is
$\leq k$. Because the degree of $G$ is $\leq1$ we get that for
any $i$ there can be at most one $j$ such that $(i,j)\in G$ hence
there are a totality of no more than $k$ edges $(i,j)\in G$ for
which (\ref{eq:oneij}) is non-zero. Hence we get
\[
\sum_{(i,j)\in G}((\id-(ij))f)(q)\leq2k|f(q)|.
\]
Since $f$ is an eigenfunction of $\lambda_{\max}$, we also have \[
\sum_{(i,j)\in G}((\id-(ij))f)(q)=\lambda_{\max}f(q)\]
and the lemma is proved.
\end{proof}
\begin{lem}
\label{lem:onestar}Let $\sigma\in\mathcal{S}_{k}$ and let $G=\str_{n,l+1}$
i.e.~a star graph with $l$ edges and the rest of the vertices isolated. Then\[
\lambda_{\max}(G;\sigma)\leq l+k\]
\end{lem}
\begin{proof}
This is a corollary of lemma \ref{lem9}.
The eigenvector for the maximal eigenvalue corresponds to an element
of the Gelfand-Tsetlin basis, which corresponds to a sequence of boxes
$(x_i,y_i)$ of $\sigma$ as in lemma \ref{lem9}. Since $\sigma$ has no
more than $k+1$ rows, we have $y_i \leq k+1$ for all $i$, and we
always have $x_i \geq 1$. According to lemma \ref{lem9},
\[
\lambda_{\max}(G;\sigma)=((l+1)-1) + (y_{l+1}-x_{l+1})\leq l+k+1-1=l+k.
\qedhere\]
\end{proof}
\begin{lem}
\label{lem:weightedstar}Let $A$ be a weighted star, i.e.\ assume
there are some $a_{2}\geq\dotsb\geq a_{n}\geq0$ such that $(1,i)$
has weight $a_{i}$ but $(i,j)$ has weight $0$ when both $i>1$
and $j>1$. Let $\sigma\in\mathcal{S}_{k}$. Then\[
\lambda_{\max}(A;\sigma)\leq2a_{2}+\dotsb+2a_{k+1}+a_{k+2}+\dotsb+a_{n}.\]
 
\end{lem}
\begin{proof}
Write $A=A_{2}+\dotsb+A_{n}$ where $A_{i}$ is a weighted star with
weights
\begin{align*}
 & \overbrace{a_{i}-a_{i+1},\dotsc,a_{i}-a_{i+1}}^{i-1\mbox{ times}},\!\!\!\overbrace{0,\dotsc,0}^{n-i\mbox{ times}} & i & =2,\dotsc,n-1\\
 & \overbrace{a_{n},\dotsc,a_{n}}^{n-1\mbox{ times}} & i & =n.
\end{align*}
Since $\lambda_{\max}$ is a norm (recall that $\sigma(\Delta_A)$ is a
positive matrix so $\lambda_{\max}$ is its norm as an $l^2$ operator),
we have that
\[
\lambda_{\max}(A;\sigma)\le\sum_{i=2}^{n}\lambda_{\max}(A_{i};\sigma).
\]
Each summand may be estimated by lemma \ref{lem:onestar}, and we
get\[
\lambda_{\max}(A_{i};\sigma)\leq(i-1+k)(a_{i}-a_{i+1}).\qquad\mbox{(define }a_{n+1}:=0\mbox{)}\]
However, for $i\leq k$ we actually get a better estimate from the
trivial bound (\ref{eq:triv}),\[
\lambda_{\max}(A_{i};\sigma)\leq2(i-1)(a_{i}-a_{i+1}).\]
Summing we get
\begin{align*}
\lambda_{\max}(A;\sigma) & \leq\sum_{i=2}^{k}2(i-1)(a_{i}-a_{i+1})+\sum_{i=k+1}^{n}(i-1+k)(a_{i}-a_{i+1})=\\
 & =\sum_{i=2}^{k+1}2a_{i}+\sum_{i=k+2}^{n}a_{i}
\end{align*}
as was to be proved.
\end{proof}

\begin{lem}\label{reduce}Let $A,H$ be two weighted graphs with $n$ vertices, and let $\sigma, \tau$ be two irreducible representations of $S_n$. If 
\begin{equation} \lambda_1(H;\tau)\geq \lambda_{\max}(H;\sigma) \label{eq:reduce}
\end{equation} and
$$ \lambda_1(A;\tau) \geq \lambda_1(A;\sigma)$$ then
$$ \lambda_1(A+H;\tau) \geq \lambda_1(A+H;\sigma).$$ \end{lem}
\begin{proof}
Recall the variational characterization of $\lambda_1$ which states
that for any positive matrix $M$, its lowest eigenvalue is the minimum
over all vectors $v$ of $\langle Mv,v\rangle$. Hence we may upper
bound $\lambda_1(A+H;\sigma)$ with any $v$, and we choose $v$ to be a unit eigenvector corresponding to $\lambda_1(A;\sigma)$. Then
\begin{align*} 
\lambda_1(A+H;\sigma) & \leq \langle(\Delta_A+\Delta_H)v,v\rangle= \\
&=\lambda_1(A;\sigma)+\langle \Delta_Hv,v\rangle \leq \lambda_1(A;\sigma)+\lambda_{\max}(H;\sigma) \leq\\
&\leq \lambda_1(A;\tau)+ \lambda_1(H;\tau) \leq 
\lambda_1(A+H;\tau).\qedhere
\end{align*}
\end{proof}
We remark that even though lemma \ref{reduce} works for any matrix $H$, we
will apply it only to matrices whose entries take two values (one of
which is 0) i.e.~to graphs whose weights are all the same.
\begin{definition} Let $G,H$ be two weighted graphs, and $\sigma, \tau$ representations of $S_n$. \begin{enumerate}
\item
We call $H$ a reducing graph for $\sigma, \tau$ if $H$ satisfies (\ref{eq:reduce}).
\item
We say that $G$ is $H$-irreducible if there does not exist a graph $H'$ isomorphic to $H$ and a number $\epsilon >0$ such that $\epsilon H' \leq G$.
\end{enumerate}
\end{definition}

The identities (\ref{eq:reverse}) and (\ref{eq:dual}) show that equation (\ref{eq:reduce}) is equivalent to $$ \lambda_{\max}(H;\sigma) + \lambda_{\max}(H;\tau\otimes\sgnrep) \leq 2\wt(H). $$
We will use this reformulation to show that a given graph $H$ is reducing.

Lemma \ref{reduce} is the basis to our strategy of reduction: In proving that $\sigma \succ \tau$, if $H$ is a reducing graph, it is enough to prove that $\lambda_1(A;\sigma) \leq \lambda_1(A;\tau)$ only for $H$-irreducible matrices $A$. Indeed, if $A$ is not $H$-irreducible, then we can find $H' \cong H$ and $\epsilon>0$ such that $A-\epsilon H'$ has nonnegative weights on the edges, and fewer nonzero weights than $A$. According to lemma \ref{reduce}, it is enough to prove the inequality for $A-\epsilon H'$. Repeating this procedure, we reduce the problem to $H$-irreducible graphs.


\begin{proof}[Proof of the theorem]
Let $n\ge 4k^2+4k$ and let $\sigma\in\mathcal{S}_k$ and $\tau\in\mathcal{S}_k\otimes\sgnrep$. The claim of
the theorem is that under these conditions $\sigma\succ\tau$.
Let $H$ be the graph with $2k$ disjoint edges i.e.~as a matrix its
coefficients $h_{ij}$ are given by
\[
h_{ij}=\begin{cases}
1&\mbox{$i=2k$ and $j=2k+1$ for some $k$}\\
0&\mbox{otherwise.}
\end{cases}
\]
By lemma \ref{lem:matching},
$$
\lambda_{\max}(H; \tau\otimes\sgnrep)+\lambda_{\max}(H;\sigma)\leq
4k=2\wt(H). 
$$
Hence, $H$ is a reducing graph for $\sigma$ and $\tau$.
It is hence enough to prove that $\lambda_1(A,\sigma) \leq \lambda_1(A, \tau)$ for any $H$-irreducible matrix $A$.

It is well known that an $H$-irreducible graph can be written as a union of $4k-2$ weighted stars. Indeed, choose
an edge $e$ of $A$ arbitrarily and remove from $A$ the two stars
centered at the two vertices of $e$. If the resulting graph is non-empty,
choose again some edge arbitrarily and remove two stars. This process
must stop after $2k-1$ steps, since otherwise we would have found $2k$
disjoint edges in $A$. Hence we wrote $A$ as a union of $4k-2$
stars. Denote 
\[
A=\sum_{i=1}^{4k-2}S_{i}.
\]
We now use lemma \ref{lem:weightedstar} for each of the $S_i$ and sum
over $i$. Recall that the $k$ edges with the largest weights of a
weighted star play a special role in lemma \ref{lem:weightedstar} ---
their weights were multiplied by 2 rather than by 1. Collecting these
special edges for the $4k-2$ stars gives a total of $k(4k-2)$ special
edges. Denote them by $e_{i}$. Thus the conclusion of lemma
\ref{lem:weightedstar} is
\[
\lambda_{\max}(A,\tau\otimes\sgnrep)\leq
  \sum_{i<j}a_{ij}+\sum_{i=1}^{k(4k-2)}a_{e_{i}}
\]
(where if $e=(i,j)$ then we denote $a_{e}=a_{ij}$). The edges $e_{i}$
combined have no more than $(k+1)\cdot(4k-2)=4k^2+2k-2$ vertices.

Let us now move to the estimate of $\lambda_1(A;\sigma)$. For that matter,
pick $2k$ vertices which do not belong to any of the $e_{i}$'s (here we use the condition $n\geq4k^{2}+4k$). Denote these vertices by $v_{1},\dotsc,v_{2k}$.
For every $i$, let $W(i)$ be the weight of $v_{i}$ i.e.\[
W(i)=\sum_{j=1}^{n}a_{v_{i}j}.\]
Assume w.l.o.g.\ that the $v_{i}$ are arranged so that $W(i)$ are
increasing. Examine again the representation $L^{2}(Q)$ from lemma
\ref{lem:Q}. By clause (2) of that lemma, since $\sigma \in
\mathcal{S}_{k}$, for any set of $n-k$ vertices there exists a nonzero
element $f\in V\subset L^{2}(Q)$, where $V$ is an invariant subspace
of $L^2(Q)$ isomorphic to $\sigma$, such that $f$ is invariant to
permutations of elements from this set. Normalize $f$ to have $||f||=1$.
We choose the set to be all vertices except
$v_{1},\dotsc,v_{k}$. Examine now 
\[
\sum_{i<j}a_{ij}(\id-(ij))f.
\]
If both $i$ and $j$ are different from $v_{1},\dotsc,v_{k}$ then
$(ij)f=f$ and the contribution to the sum is $0$. Otherwise we simply
estimate\[
||a_{ij}(\id-(ij))f||\leq2a_{ij}\]
(here $||\cdot ||$ is the norm in $L^2(Q)$) and we get
\[
\left\Vert \sum_{i<j}a_{ij}(\id-(ij))f\right\Vert \leq2\sum_{i=1}^{k}W(i)\leq\sum_{i=1}^{2k}W(i)\]
where the second inequality comes from the fact that we chose the
$W(i)$ increasing. This bounds $\lambda_1(A;\sigma)$ and we get
\[
\lambda_1(A;\sigma)+\lambda_{\max}(A;\tau\otimes\sgnrep)\leq\sum_{i=1}^{2k}W(i)+\sum_{i<j}a_{ij}+\sum_{i=1}^{k(4k-2)}a_{e_{i}}\]
Since the vertices $v_{1},\dotsc,v_{2k}$ are different from the vertices on the edges $e_{i}$, the above sum contains each edge no more than twice. Hence,
\[ \lambda_1(A; \sigma) \leq 2\wt(A)-\lambda_{\max} (A; \tau\otimes\sgnrep) = \lambda_1(A;\tau). \qedhere\]
\end{proof}

\appendix

\section{\label{sec:gamma}The hook-shaped diagrams}

In this appendix we prove the claim appearing in the introduction that 
\begin{equation}\label{eq:Bacher}
[n]\succ [n-1,1] \succ \dotsb \succ [1^n].
\end{equation}
This is basically a result of Bacher \cite{B94}, 
who showed that the eigenvalues $\lambda_i(A;[n-k,1^k])$ are simply
all the sums of all $k$-tuples of the eigenvalues
$\lambda_i(A;[n-1,1])$. This immediately implies (\ref{eq:Bacher})
(recall that the eigenvalues are all non-negative). However, he used a
different description of these representations, as 
wedge products of $[n-1,1]$. We will now prove that the two
descriptions coincide. This was known before (for example, it is
mentioned without proof in \cite{CLR09} in the penultimate paragraph
of the introduction) but we found no proof in the literature.

\begin{lem} $ \wedge^k[n-1,1]\cong[n-k,1^k] $
\end{lem}

The proof will use the Murnaghan-Nakayama formula for the characters of
the irreducible representations of $S_n$. See e.g.~\cite[Theorem 4.10.2]{S00}. We will not give
a full description of this formula here. 

\begin{proof}
Recall that for a representation $\rho$ the character is a function
$\chi:G\to \C$ defined by $\chi(g)=\tr \rho(g)$ and that two
representations are isomorphic if and only if their characters
coincide \cite[Corollary 1.9.4(5)]{S00}. Thus it is enough to prove that the two representations have the same character.
Let us denote the character of the representation on the left hand
side by $\chi_k^\wedge $ and the right hand side by $\chi_k^\Gamma$
(the letter $\Gamma$ reminds us of a hook). We will prove that 
$\chi_k^\wedge+\chi_{k-1}^\wedge=\chi_k^\Gamma+\chi_{k-1}^\Gamma$.
Since the lemma is true for $k=1$, that will be enough.

Let $V$ be the $n$-dimensional Euclidean space, with the standard
basis $e_1,..,e_n$, viewed as the standard representation of $S_n$
(see page \pageref{pg:stdrep}). Since 
$V \cong [n-1,1] \oplus [n]$, and $[n]$ is one dimensional, it can be
easily seen that 
$$\wedge^k V \cong \wedge^{k-1}[n-1,1] \oplus \wedge^k[n-1,1].$$ 
Therefore $\chi_{\wedge^kV}=\chi_k^\wedge + \chi_{k-1}^\wedge$. Recall now the
standard basis for $\wedge^k V$: for any
subset $K=\{i_1<i_2<\dotsb<i_k\}$ of $\{1,2,\dotsc,n\}$ let
$e_K=e_{i_1} \wedge \dotsb \wedge e_{i_k}$. We will calculate
the trace of a permutation $g$ (acting on $\wedge^k V$) using this
basis. Let $g\in S_n$ have cycles of lengths $c_1\geq c_2 \geq \dotsb \geq c_r$. We have $ge_K=\pm e_{g(K)}$. 
The only
contribution to the trace comes from $e_K$'s for which $g(K)=K$, and
in this case the $\pm$ above is simply $\sgn(g|_K)$. Hence
we get 
$$ \chi_{\wedge^kV}(g)=\sum_{\epsilon_1,\dotsc,\epsilon_r \in \{0,1\}: \sum \epsilon_i c_i = k} (-1)^{\sum \epsilon_i (c_i-1)}. $$
Let us now calculate $\chi_k^\Gamma+\chi_{k-1}^\Gamma$. We use the Murnaghan-Nakayama rule, which takes a nice form for hook-shaped diagrams:
For a hook-shaped diagram $\gamma$, let $\bS(\gamma)$ be the
set of all sequences of Young diagrams
$\gamma=\gamma_1,\dotsc,\gamma_{r+1}=\emptyset$ such that for $1 \leq
i\leq r$, $\gamma_{i+1}$ is obtained from $\gamma_i$ by removing $c_i$
consecutive boxes. Note that except for the last stage, all the
removed parts do not contain the corner $(1,1)$ and hence are either
horizontal or vertical bars. We will call the removed parts of the
$r-1$ first stages the \emph{bars of the sequence}. For any set of
consecutive boxes we define its \emph{height} to be the number of rows it
occupies. And for an element $(\gamma_1,\dotsc,\gamma_{r+1})$ of
$\bS(\gamma)$, we define its height to be $\sum_{i=1}^r
(1+\height(\gamma_i \setminus \gamma_{i+1}))$. Then, according to the
Murnaghan-Nakayama rule, 
$$ \chi_k^\Gamma(g) = \sum_{s\in \bS([n-k,1^k])} (-1)^{\height(s)}. $$
Let $f$ be the bijective function from the Young
diagram $[n-k,1^k]$ to the Young diagram $[n-k+1,1^{k-1}]$ taking the
box at position $(1,i)$ to the box $(1,i-1)$ for all $i>1$, and taking
the box at $(i,1)$ to $(i+1,1)$ for all $i$. See figure \ref{fig:f}.
\begin{figure}
\input{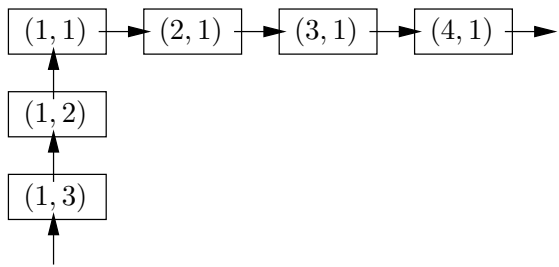}
\caption{\label{fig:f}The mapping $f:[n-k,1^k]\to[n-k+1,1^{k-1}]$.}
\end{figure} 
Let $\bA$ be the subset of $\bS([n-k,1^k])$ of all sequences for which
no bar has $(1,2)$ as an endpoint, and let $\bB$ be the subset of
$\bS([n-k+1,1^{k-1}])$ of all sequences for which no bar has $(2,1)$
as an endpoint. Let $F:\bA\rightarrow \bB$ be the function defined by
applying $f$ to each stage of the sequence. Then $F$ is well defined
because the condition that $(1,2)$ is not an end point ensures that
the image is indeed a set of legal Young diagrams; and $F$ is a bijection, since an inverse function can be defined using $f^{-1}$. Moreover, the corresponding terms for $s$ and $F(s)$ cancel out in the sum for $\chi_k^\Gamma+\chi_{k-1}^\Gamma$.
Hence, $$\chi_k^\Gamma+\chi_{k-1}^\Gamma= \sum_{s\in
  \bS([n-k,1^k])\setminus \bA} (-1)^{\height(s)}+\sum_{s\in
  \bS([n-k+1,1^{k-1}])\setminus \bB} (-1)^{\height(s)}.$$ But this is equal
to $\chi_{_V}(g)$: A sequence $\epsilon_1,..,\epsilon_r\in \{0,1\}$
such that $\sum \epsilon_i c_i=k$ determines a unique term in one of
the above two summands in the following way: at each stage, we remove
a vertical bar of size $c_i$ if $\epsilon_i=1$ and a horizontal one if
$\epsilon_i=0$. We get a sequence of Young
diagrams in the first summand when $\epsilon_r=0$, and in the second
summand when $\epsilon_r=1$. Further, it is easy to check that all terms
in both summands are obtained in this way, and that the
corresponding term is $(-1)^{\sum {\epsilon_i (c_i-1)}}$.
\end{proof}


\begin{thebibliography}{99}
\bibitem{AK10}Gil Alon and Gady Kozma, {\em The probability of long
  cycles in the interchange process}. Preprint (2010). 
  \href{http://arxiv.org/abs/1009.3723}{\nolinkurl{1009.3723}}

\bibitem{B94}Roland Bacher, {\em Valeur propre minimale du laplacien
  de Coxeter pour le groupe sym\'etrique}. [French, Minimal eigenvalue
  of the Coxeter Laplacian for the symmetric group]  J. Algebra  {\bf
  167:2} (1994), 460--472. 

\bibitem{CLR09} Pietro Caputo, Thomas M. Liggett and Thomas
  Richthammer,  {\em Proof of Aldous' spectral gap
    conjecture}. J. Amer. Math. Soc. {\bf 23:3} (2010),
  831--851. 


\bibitem{C09a} Filippo Cesi, {\em On the eigenvalues of Cayley graphs
  on the symmetric group generated by a complete multipartite set of
  transpositions}. J. Algeb. Combin. {\bf 32:2} (2010),
  155--185. 


\bibitem{D09} A. B. Dieker, {\em Interlacings for random walks on
  weighted graphs and the interchange process}. SIAM J. Discrete
  Math. {\bf 24:1} (2010),
  191--206. 

\bibitem{DS81}Persi Diaconis and Mehrdad Shahshahani, 
{\em Generating a random permutation with random transpositions}.
Z. Wahrsch. Verw. Gebiete {\bf 57:2} (1981), 159--179. 

\bibitem{FH91} William Fulton and Joe Harris, \emph{Representation theory. A
first course}. Graduate Texts in Mathematics, 129. Readings in
Mathematics. Springer-Verlag, New York, 1991. 

\bibitem{G93} David M. Goldschmidt, \emph{Group characters, symmetric
functions, and the Hecke algebra}. University Lecture Series,
4. American Mathematical Society, Providence, RI, 1993.

\bibitem{HJ96}Shirin Handjani and Douglas Jungreis, {\em Rate of
  convergence for shuffling cards by transpositions}.
  J. Theoret. Probab. {\bf 9:4}  (1996), 983--993. 

\bibitem{JK81} Gordon James and Adalbert Kerber, \emph{The
representation theory of the symmetric group. With a foreword by
P. M. Cohn. With an introduction by Gilbert de
B. Robinson}. Encyclopedia of Mathematics and its Applications,
  16. Addison-Wesley Publishing Co., Reading, Mass., 1981.

\bibitem{M06} Alexander I. Molev, \emph{Gelfand-Tsetlin bases for classical
Lie algebras}.  Handbook of algebra. Vol. 4,  109--170,
Handb. Algebr., 4, Elsevier/North-Holland, Amsterdam, 2006. 

\bibitem{M08} Ben Morris, {\em Spectral gap for the interchange
  process in a box}.  Electron. Commun. Probab.  {\bf 13}  (2008),
  311--318. 

\bibitem{OV96} Andrei Okounkov and Anatoly Vershik, \emph{A new
approach to representation theory of symmetric groups}.  Selecta
Math. (N.S.)  {\bf 2:4}  (1996),  581--605. 

\bibitem{S00} Bruce E. Sagan, {\em The symmetric
  group. Representations, combinatorial algorithms, and symmetric
  functions}. Second edition. Graduate Texts in Mathematics,
  203. Springer-Verlag, New York, 2001. 

\bibitem{SC08} Shannon Starr and Matt Conomos, {\em Asymptotics of the
  Spectral Gap for the Interchange Process on Large
  Hypercubes}. Preprint (2008), 

\end{thebibliography}
\end{document}